\documentclass[10pt,a4paper,oneside]{amsart}

\usepackage{xcolor}
\usepackage{bbm}

\newtheorem{thm}{Theorem}[section]

\newtheorem{prop}[thm]{Proposition}
\newtheorem{conj}[thm]{Conjecture}

\newtheorem{rmk}[thm]{Remark}

\theoremstyle{remark}
\newtheorem*{ackno}{Acknowledgements}

\def\dd{\mathrm{d}}
\def\R{\mathbb{R}}
\def\Z{\mathbb{Z}}

\newcommand{\mycomment}[1]{}

\title[]{Leading order asymptotics for non-local energies and the Read-Shockley law}

\author[]{Peter J. Grabner}
\address[Grabner]{Institute of Analysis and Number Theory, 
Graz University of Technology, 
Kopernikusgasse 24/II, 
8010 Graz, 
Austria}
\email{peter.grabner@tugraz.at}

\author[]{Florian Theil}
\address[Theil]{Mathematics Institute, Warwick University, CV4 7AL Coventry, UK}
\email{f.theil@warwick.ac.uk}
\begin{document}

\begin{abstract}
    We study an energy minimization problem $\sum_{i \neq j} W(z_i - z_j)$ for $N$ points $\left\{z_1, \dots, z_N\right\}$ with applications in dislocation theory. The $N$ points lie in the two-dimensional domain $\mathbb{R} \times [-\pi, \pi]$, 
    where the kernel $W$ is derived from the Volterra potential $V(x,y) = \frac{x^2}{x^2+y^2}-\frac12\log(x^2+y^2)$. We prove that the minimum energy is given by $- N \log{N} +\mathcal{O}(N)$. This lower bound recovers the leading order term of the Read-Shockley law characterizing the energy of small angle grain boundaries in polycrystals.
\end{abstract}
\maketitle

\section{Introduction}
The classical Read-Shockely formula 
$$ \gamma = -\gamma_0\, \theta \log \theta + {\mathcal O}(\theta),  \quad 0\leq \theta \ll 1$$
expresses the energy density $\gamma$ of small angle grain boundaries in terms of the tilt angle $\theta$ and the constant $\gamma_0= \frac{\mu b}{4\pi(1-\nu)}$ where $b$ is the length of the Burgers vector, $\nu \in [0,1)$ is the Poisson ratio and 
$\mu >0$ is the shear modulus. 
In~\cite{Read-Shockley+1950} this formula is obtained as the leading order term for a straight equi-spaced dislocation array in a two-dimensional isotropic linearly elastic medium. This derivation does not address the question whether such configurations are energetically optimal.

Recently there has been renewed interest in detailed mathematical modeling of grain boundaries, e.g. \cite{QiuaSrolovitzRohrerHanaSalvalaglio2025}. The analysis of phenomena like grain boundary motion involves careful use of grain boundary energy models,  see \cite{HanSrolovitzSalvalaglio2022}. Therefore, it is desirable to derive the Read-Shockley formula from variational principles.

The first step in this direction has been taken by
Luckhaus and Lauteri; in~\cite{Lauteri-Luckhaus:2016} they analyzed a non-linear variational continuum model
$$ {\mathcal F}_\epsilon(A,S) =\frac1\epsilon \left( \int_{\Omega \setminus S_\epsilon} \mathrm{dist}^2(A,\mathrm{SO}(2))\, \dd x + |S_\epsilon|\right),$$
where $\Omega = [-1,1]^2 \subset \mathbb{R}^2$ is a square, $S \subset \Omega$, $S_\epsilon  = S+B_\epsilon(0)$, $A \in L^1(\Omega, \R^{2\times 2})$ and the pair $(A,S)$ satisfies admissibility conditions. 
Their main result is matching upper and lower bounds for the infimum of $\mathcal F_\epsilon$ which are uniform in $\epsilon$ and consistent with the Read-Shockley scaling. More precisely, it is shown that
\begin{align}
\label{eq.RS-scaling}
 C_1 \leq \frac{\inf \mathcal F_\epsilon(A,S)}{\theta \log \theta}\leq C_2
\end{align}
if
$A$ satisfies the boundary condition
$$ A|_{x=\pm 1} = \begin{pmatrix}
\cos \theta & \pm \sin \theta\\
\mp \sin \theta & \cos \theta
\end{pmatrix}.$$ 
It is worth mentioning that in \cite{Fortuna-Garroni-Spadaro-2025} it is shown that $\mathcal F_\epsilon$  has a $\Gamma$-limit for $\epsilon\to0$ which can be written in the form of a grain boundary energy density. 

Currently, it is not known, if it is possible to choose $C_1 = C_2$ in~\eqref{eq.RS-scaling} as $\epsilon, \theta \to 0$. To address this problem we revisit the quadratic setting where the Read-Shockley formula has been derived. Our main result is a rigorous proof that the formula holds even if the dislocations are not necessarily equispaced, or on a straight line (Theorem~\ref{thm.ReadShockley}).

The paper is organized as follows. In Section~\ref{sec.Read-Shockley} we recall the dislocation model introduced by Read and Shockley \cite{Read-Shockley+1950} and present it in a variational formulation, which allows then to put it in the context of discrete energy minimization where one seeks to minimize finite sums of the form
\begin{equation*}    {\mathcal E}(\omega) = \sum_{\substack{z,z'\in\omega\\z\neq z'}}K(z-z').
\end{equation*}
Here $\omega$ is a finite subset of the underlying space ($\Omega=\mathbb{R}\times\mathbb{T}$) and $K:\Omega\times\Omega\to\mathbb{R}$ is a kernel function; we set $N=|\omega|$. The kernel function $K$ is singular at $\{0\}$ which prevents the use of Bochner's theorem. The key step in our approach is the construction of a positive definite regularization $K_t \in L^\infty(\Omega)$ so that $K-K_t \geq o(1)$ for $0<t\ll 1$.

Taking a sequence of point sets $\omega$ with $\frac1N\sum_{z\in\omega}\delta_z\rightharpoonup\mu$ and rescaling with $N^{-2}$ the corresponding discrete energies converge in the sense of $\Gamma$-limits to the continuous energy
\begin{equation*}
    E_K(\mu)=\iint\limits_{\Omega\times \Omega}K(z-z')\,\dd\mu(z)\,\dd\mu(z'),
\end{equation*}
where the measure $\mu$ ranges through all Borel probability measures on $\Omega$, cf.~\cite{Mora-Rondi-Scardia2019}. 
Especially, the empirical measures of point configurations $\omega$ minimizing the discrete energy converge weakly to a minimizer of the continuous energy (see also \cite[Chapter~4]{Borodachov_Hardin_Saff2019:discrete_energy}). 

Continuous energies are the object of study of classical potential theory (e.g. \cite{Landkof1972:potential_theory}). An important general result is the fact that under the condition of strict positive definiteness of the kernel $K$ and compactness of the space $\Omega$ the empirical measures given by the discrete minimizers weakly tend to the unique continuous minimizer. In our case we neither have strict positive definiteness nor compactness of the space, which forces us to use a different methodology. 

In Section~\ref{sec:cont-energ} we study the continuous energy problem for the Read-Shockley kernel and observe that without an additional confining external potential the continuous minimizer of the energy is not unique; we are still able to characterize all minimizers.
In particular, it turns out that minimizers are not necessarily concentrated on a one-dimensional set (Proposition~\ref{prop:cont-no-ext}). On the other hand, the concavity of the Read-Shockley energy with respect to $\theta$ promotes grain-boundaries with a one-dimensional support. This shows that the discrete model measures genuinely finer properties than the continuous $\Gamma$-limit. 

We also study the continuous energy in the presence of an external potential proportional to $|x|$ (the $x$-coordinate of the point) and determine a phase transition between non-existence and uniqueness of the minimizer depending on the strength of the external field.

\section{The Read-Shockley formula}\label{sec.Read-Shockley}
\subsection{Derivation of the variational model}
We work in the setting of two-dimensional, linear  elasticity. 
The elastic and strain fields $\sigma:\R^2 \to  \R^2_\mathrm{sym}$ and $\beta:\R^2 \to \R^{2\times 2}$ corresponding to a collection of
edge dislocations with Burgers vector ${\bf b} = (b,0)$ and cores located in $\omega \subset \R^2$ satisfy the equilibrium equations 
\begin{align}
\label{eq.equilibrium-equations}
\begin{cases}
\nabla \cdot \sigma &= {\bf 0},\\
\mathrm{curl}(\beta) &= \sum\limits_{z \in \omega} \delta_{z}{\bf b},
\end{cases}
\end{align}
in the sense of distributions. See~\cite{cermelli-leoni2006, GarroniLeoniPonsiglione2008} for a discussion of the model and the properties of solutions of~\eqref{eq.equilibrium-equations}.

Analytic expressions for $\sigma$ and $\beta$ are well known in the case of isotropic elasticity where $\sigma$ and $\beta$ satisfy the constitutive relation 
$\sigma = \frac{2\mu \nu}{1-2\nu}\mathrm{tr}(\beta)\mathrm{Id}+\mu (\beta+\beta^T)$.
Define
\begin{align*} 
u_0 =&\frac{-b}{8\pi(1-\nu)} \begin{pmatrix} 4(1-\nu)\arctan (x,y)-\frac{2xy}{r^2}\\
(1-2\nu)\log(x^2+y^2)+\frac{x^2-y^2}{r^2}
\end{pmatrix},\\[1em]
\sigma_0 =& \frac{2\gamma_0}{r^4}\begin{pmatrix} -y(3x^2+y^2) & x(x^2-y^2) \\ x(x^2-y^2) & y(x^2-y^2)
\end{pmatrix},
\end{align*}
where $r^2=x^2+y^2$ and $\arctan(x,y)\in (-\pi,\pi]$ denotes the argument of $x+iy$ with a branch cut along $\mathcal C=\{(0,y)\;:\; y<0\}\subset \R^2$, cf.
\cite[eqns. (3.43), (3.45), (3.46)]{Hirth-Lothe+1982:theory-of-dislocations}.
Then $\beta = \sum_{z \in \omega}\nabla u_0(\cdot - z)$, $\sigma = \sum_{z \in \omega}\sigma_0(\cdot - z)$ satisfy~\eqref{eq.equilibrium-equations}. Since $u_0$ has a jump discontinuity along $\mathcal C$ we need to clarify that $\nabla u_0$ is defined to be the regular part of the distributional derivative, i.e. it ignores the jump.

The elastic energy density associated with $\sigma$ and $\beta$ is given by 
$$ f(\zeta)=  \frac{1}{2} \beta(\zeta):\sigma(\zeta),$$
where $\beta:\sigma = \mathrm{trace}(\beta^T \sigma) = \sum_{1\leq i, j \leq 2} \beta_{ij}\sigma_{ij}$. 
Note that $f$ is singular on $\omega$; 
the total elastic energy stored in a compact domain $\Omega$ such that $\Omega\cap \omega = \emptyset$ is given by 
$$ \mathcal E = \int_{\Omega}f(\zeta)\, \dd \zeta.$$

Since $f$ is a bilinear function in $\beta$ and $\sigma$ we can decompose it into pair-interactions
\begin{align*}
f(\zeta)=&\frac12 \left(\sum_{z \in \omega}\beta_0(\zeta-z)\right): \left(\sum_{z' \in \omega} \sigma_0(\zeta-z')\right)
= \sum_{z,z'\in \omega} q_{z,z'}(\zeta),
\end{align*}
where
$$q_{z,z'}(\zeta) = \frac12\beta_0(\zeta-z):\sigma_0(\zeta-z').$$
Now we consider the energy density created by periodic configurations $\widehat {\omega} \subset \R^2$ satisfying $\widehat {\omega} = \bigcup_{n \in \Z}(\omega+2\pi n h e_2) $, where $h>0$, $e_2=(0,1)$ and $\omega \subset \Omega_h = \R \times [-\pi h , \pi h)$.
Let
$$q^\mathrm{per}_{z,z'}(\zeta)= \frac12\sum_{n,n' \in \Z}\beta(\zeta-z-2\pi n h\, e_2):\sigma(\zeta-z'-2\pi n' h\,e_2),$$
then
\begin{align} \label{eq.toten}
 \sum_{z,z'\in \widehat {\omega}}\int_{\Omega} q_{z,z'}(\zeta)\, \dd \zeta = \sum_{z,z' \in \omega}\int_{\Omega} q_{z,z'}^\mathrm{per}(\zeta)\,\dd \zeta
\end{align}
as long as $\Omega$ is compact and $\omega \cap \Omega= \emptyset$. With this notation the normalized total energy stored in $\Omega_h = \R \times [-\pi h,\pi h]$ admits the representation
\begin{align} \label{eq.pairsumprep}
{\mathcal E}(\omega) = \sum_{z \in \omega}U_\mathrm{self}(z)+ \sum_{z\neq z' \in\omega} U_\mathrm{int}(z-z')
\end{align}
where
\begin{align} \label{eq.Uself}
U_\mathrm{self}(z) &= \frac1{2\pi h}\int_{\Omega_h \setminus B_\rho(z)} q^\mathrm{per}_{z,z}(\zeta) \, \dd \zeta,\\
\label{eq.Uint}
U_\mathrm{int}(z,z') &= \frac1{2\pi h}\int_{\Omega_h} q^\mathrm{per}_{z,z'}(\zeta)\, \dd \zeta.
\end{align}
To avoid the complication that $q_{z,z} \not \in L^1(\Omega_h)$
we have removed $B_\rho(z)$ from the domain of integration in $U_\mathrm{self}$, see~\cite{cermelli-leoni2006} for a discussion of this step. The radius $\rho$ is sometimes referred to as `core radius' in the literature, its purpose is to account for atomistic effects. The precise value only affects the infimum of $\mathcal E$ at non-leading order when $N =o(h)$, $h\gg 1$. For the sake of notational simplicity we will choose $\rho =b$. 

\subsection{Main results}
Now we are in a position to provide a variational derivation of the leading order constant in the Read-Shockley formula. 
The key property in grain boundary modeling is the density of dislocations per unit-length. The tilt $\theta$ angle between two adjacent grains satisfies the relation 
$$ \sin \theta = \frac{b}{D},$$
where $D$ is the average spacing between dislocations. In our notation $D=\frac{2\pi h}N$, so for small angle grain boundaries where $N={\mathcal O}(h)$ one obtains the approximation
$$ \theta = \frac{N b}{2\pi h} \text{ for } 0<\theta \ll 1.$$
\begin{thm}[Read-Shockley formula]
\label{thm.ReadShockley}
Let $\mathcal E$ be given by~\eqref{eq.pairsumprep}. Then
\begin{align}
\label{eq.RSequation}
\min_{|\omega_N| = N} {\mathcal E}(\omega) = -\gamma_0\,\theta \log \theta+ {\mathcal O}(\theta), \quad 0< \theta \ll 1, \end{align}
where $\gamma_0=\frac{\mu b}{4\pi (1 - \nu)}$.
\end{thm}
The proof of~Theorem~\ref{thm.ReadShockley} and other results in the article crucially rely on an analytic representation of $U_\mathrm{int}$. This result is well known, we include it here for the convenience of readers.
\begin{prop} \label{prop.derivation}
Let  
\begin{align} 
\label{eq.Walpha}
 W_\alpha(x,y)= \frac{1}{2}\left[\frac{\alpha\,x \sinh x}{\cosh x-\cos y}-\log(2(\cosh x-\cos y)) +(1-\alpha)|x|\right].
\end{align}
The functions $U_\mathrm{self}$ and $U_\mathrm{int}$ admit the representation
\begin{align}
\label{eq.self} U_\mathrm{self}(z) =& \frac{\gamma_0 b}{2\pi h}(\log h+{\mathcal O}(1)), \quad h\gg 1,\\ 
\label{eq.int} U_\mathrm{int}(z,z')=&\frac{\gamma_0 b}{2\pi h} \,W_1((z-z')/h).
\end{align}

\end{prop}
Our main result is a  lower bound for the interaction energy.
\begin{prop} 
\label{prop.mainlowerbound}
Let $C=1-\log 2<0.31$.
For each $\omega \subset \R \times (-\pi,\pi]$ with the property that $|\omega|=N$ the inequality
\begin{equation} \label{eq.mainlb}
\sum_{z\neq z' \in \omega} W_\alpha(z-z') \geq -N(\log N +C)
\end{equation}
holds.
\end{prop}
\begin{conj}
Proposition~\ref{prop.mainlowerbound} is true for $C=0$. This would imply that the minimum energy is attained for equally spaced points on a circle $x=x_0$. Numerical experiments support this conjecture.
\end{conj}
With these results we are in a position to justify the Read-Shockley formula.
\subsection{Proof of Theorem~\ref{thm.ReadShockley} and Propositions~\ref{prop.derivation} and~\ref{prop.mainlowerbound}}
\begin{proof}[Proof of Thm.~\ref{thm.ReadShockley}]
We will establish lower and upper bounds which match at leading order.

{\bf Lower bound}

If $\omega \subset \Omega_h$ has the property that $|\omega|=N$ then Prop.~\ref{prop.derivation} implies that
\begin{align*}
{\mathcal E}(\omega)=& \frac{\gamma_0b}{2\pi h}\left(\sum_{z\in  \omega} (\log h +{\mathcal O}(1))+\sum_{z \neq z' \in \omega} W_1((z-z')/h)\right)\\
=& \frac{\gamma_0b}{2\pi h}\left(N (\log h+{\mathcal O}(1)) +\sum_{z \neq z' \in \omega} W_1((z-z')/h)\right).
\end{align*}
After rescaling $\omega$ the second term reads $ \sum_{z \neq z' \in \omega} W_1(z-z')$ and we can apply Prop.~\ref{prop.mainlowerbound}. Therefore, 
\begin{align*}
{\mathcal E}(\omega)\geq \gamma_0 \frac{Nb}{2\pi h} \left(
\log\left(\frac{2\pi h}{b}\right)-\log\left(\frac{2\pi}b\right)-\log N-C\right).
\end{align*}
Recalling that $\theta=\frac{Nb}{2\pi h}$ we obtain the desired lower bound
\begin{align*}
{\mathcal E}(\omega)\geq-\gamma_0\,\theta\left((\log \theta+C+\log(2\pi/b)\right).
\end{align*}
{\bf Upper bound}

We will demonstrate that 
\begin{align} \label{eq.ub}
 \sum_{z\neq z' \in \omega_N} W_\alpha(z-z')=-N\log N
\end{align}
if $\omega_N$ is an equispaced vertical  configuration, i.e. 
$$\omega_N = \{ (x_0,2\pi j/N) \; : \; j = 0\ldots N-1\},$$
where $x_0 \in \R$ is arbitrary.
This result can also be found in~\cite{Sandier-Serfaty+2014:log-gases}. For the convenience of the reader we include a simple, self-contained proof here.

Observe that the first term in $W_\alpha$ vanishes if $x=0$ therefore we only have to analyze the logarithmic term.
\begin{align}
\label{eq.equirep}
\sum_{j=1}^N\frac12\sum_{\ell=1}^N  {\bf 1}_{\ell\neq j}\left[\log(1-\cos (2\pi(\ell-j)/N))+\log2\right]
=  \frac{N}{2}\log(F_N)
\end{align}
where 
\begin{align*}
F_N = &\prod_{j=1}^{N-1}[2(1-\cos(2j\pi/N))]=\prod_{j=1}^{N-1}\left[4\sin(\pi j/N)^2\right]
=\left[\prod_{j=1}^{N-1}(2 \sin(\pi j/N))\right]^2
\end{align*}
The sine-product formula
\begin{align} \label{eq.sinprod}
 \prod_{j=0}^{N-1}\sin(x+\pi j/N)= 2^{1-N} \sin(Nx)\end{align}
implies that 
$$F_N= \left(\lim_{x \to 0} \frac{\prod_{j=0}^{N-1} (2\sin(x+\pi j/N))}{2 \sin(x)}\right)^2 =\left(\lim_{x\to 0} \frac{\sin(Nx)}{\sin(x)}\right)^2 = N^2$$ 
and therefore the right hand side of~\eqref{eq.equirep} equals $- N\log N$ which is~\eqref{eq.ub}.

\end{proof}
\begin{proof}[Proof of Prop.~\ref{prop.derivation}]

We prove the representation of the periodized interaction kernel.

\medskip\noindent
\emph{Step 1. Periodization identities.}
To simplify the task the individual terms are periodized separately. We will show that the following identities hold.
\begin{align}
\sum_{n\in\mathbb Z}\frac{1}{x^2+(2\pi n+y)^2}
&=\frac{1}{2x}\frac{\sinh x}{\cosh x-\cos y}, \label{eq:perid1}\\
\label{eq:reglog}
\log(x^2+y^2)+\sum_{n\in \Z \setminus\{0\}}\log\left[\frac{x^2+(2\pi n+y)^2\big)}{(2\pi n)^2}\right]
&=\log\left(2(\cosh(x)-\cos(y))\right).
\end{align}
To show \eqref{eq:perid1} we start with the classical expansion
\begin{equation*}
    \pi\cot\left(\frac{y\pm ix}2\right)=\frac{2\pi(y\mp ix)}{x^2+y^2}+\sum_{n\in\mathbb{Z}\setminus\{0\}}
    \left(\frac{2\pi}{(y-2\pi n)\pm ix}-\frac1n\right).
\end{equation*}
Subtracting the two equations with the two choices of the sign gives
\begin{equation*}
    \pi\cot\left(\frac{y- ix}2\right)-\pi\cot\left(\frac{y+ ix}2\right)=\frac{2\pi i\sinh(x)}{\cosh(x)-\cos(y)}=
    4\pi i\sum_{n\in\mathbb{Z}}\frac{x}{x^2+(y-2\pi n)^2},
\end{equation*}
which gives \eqref{eq:perid1}.

For the proof of \eqref{eq:reglog} we start with the classical product expansion for the sine function again evaluated at $\frac{y\pm ix}{2\pi}$
\begin{align*}
    2(\cosh(x)-\cos(y))&=4\sin\left(\frac{y+ix}2\right)\sin\left(\frac{y-ix}2\right)\\&=
    (x^2+y^2)\prod_{n=1}^\infty\left(1-\frac{(y+ix)^2}{(2\pi n)^2}\right)\left(1-\frac{(y-ix)^2}{(2\pi n)^2}\right).
\end{align*}
Regrouping the factors in the product gives
\begin{equation*}
    2(\cosh(x)-\cos(y))=(x^2+y^2)\prod_{n=1}^\infty\frac{(x^2+(y+2\pi n)^2)(x^2+(y-2\pi n)^2)}{(2\pi n)^4}.
\end{equation*}
Taking the logarithm now gives \eqref{eq:reglog}.

\medskip\noindent
\emph{Step 2. Periodized stress.}
Define the Volterra potentials \begin{align} \label{eq.Valpha}
V_\alpha(x,y) =\alpha \frac{x^2}{x^2+y^2} -\frac12\log(x^2+y^2).
\end{align} Differentiation shows that $\sigma$ admits the representation
\begin{align*}
\sigma = 2\gamma_0\begin{pmatrix} \partial_y V_1 & -\partial_x V_1,\\
-\partial_x V_1& \partial_y V_{-1} \end{pmatrix}.
\end{align*}
Equations~\eqref{eq:perid1} and~\eqref{eq:reglog} imply that for $(x,y) \in \Omega_h$
\begin{align} \label{eq.Valphaper}
V_\alpha(x,y)+\sum_{n \in \Z\setminus \{0\}} \left[V_\alpha(x,y+2\pi h n) - \log(2\pi n)\right]=& W_\alpha(x/h,y/h).
\end{align}
Formula~\eqref{eq.Valphaper} 
implies that
\begin{align} \label{eq.sigmaper}
\sigma_\mathrm{per}=\sum_{n\in \Z}\sigma(x,y+2\pi n h) = \frac{2\gamma_0}{h}\begin{pmatrix}
\partial_y W_1& -\partial_x W_1 \\
-\partial_x W_1 & \partial_y W_{-1}
\end{pmatrix}(z/h).
\end{align}

Uniform convergence away from the singular set justifies differentiation
under the sum.
The periodized strain $\beta_\mathrm{per}$ is defined in a similar fashion.

\medskip\noindent
\emph{Step 3. Reduction to a line integral.}
Recall that $\beta = \sum_{z \in \omega}\nabla u_0(\cdot -z)$ where $u_0$ has a jump discontinuity along the branch cut $\mathcal C$. Since $\nabla\!\cdot\!\sigma_{\mathrm{per}}=0$ (cf~\eqref{eq.equilibrium-equations}) we can reduce the the two-dimensional integrals in~\eqref{eq.Uself} and \eqref{eq.Uint} to a one-dimensional integrals along $\mathcal C$ using partial integration.

If $u^\pm$ denotes the limiting values of $u$ when $\zeta$ approaches $\mathcal C$ from above (+) or below (-) then $u^--u^+ = (b,0)$, this follows immediately from the requirement that $\mathrm 
{curl} \nabla u_0=\delta_0 {\bf b}$ (eqn.~\eqref{eq.equilibrium-equations}).
Integrating the right-hand side of~\eqref{eq.Uint} by parts implies 
\begin{align*}
U_\mathrm{int}(z,z') =& \frac1{4\pi h} \int_{\Omega_h \setminus {\mathcal C}}\beta_\mathrm{per}(\zeta-z):\sigma_\mathrm{per}(\zeta-z')\, \dd \zeta\\
=&\frac{1}{4\pi h}\int_0^\infty e_2\cdot\sigma_\mathrm{per}(z-z'-(0,x)) \underbrace{[u_\mathrm{per}^--u_\mathrm{per}^+]}_{=(b,0)}\, \dd x\\
\stackrel{\eqref{eq.sigmaper}}=& -\frac{\gamma_0 b}{2\pi h} \int_0^\infty \partial_x W_1(z-z'-(x,0))\, \dd x=\frac{\gamma_0 b}{2\pi h} W_1(z-z'),
\end{align*}
which is~\eqref{eq.int}. 

\medskip\noindent
\emph{Step 4. Self-energy.}
The integration domain in~\eqref{eq.Uself} is slightly more involved because we have to remove the non-integrable singularity.
\begin{align*}
U_\mathrm{self}(z) =&\frac1{4\pi h}\int_{\Omega \setminus ({\mathcal C} \cup B_\rho)} \beta_\mathrm{per}:\sigma_{\mathrm{per}}\, \dd \zeta\\
= &\frac{1}{4\pi h}\biggl[\int_{\mathcal C\setminus B_\rho}  e_2\cdot \sigma_\mathrm{per} {\bf b}\, \dd x- \rho\int_{S^1} \xi\cdot [\sigma_\mathrm{per}  u_\mathrm{per}](\rho \xi) \, \dd \xi\biggr]\\
=&\frac{\gamma_0 b}{2\pi h^2}\int_\rho^\infty \partial_x W_\alpha(-x/h,0)\, \dd x - 
O\left(\frac{\rho b}{h} \log(\rho)\right)\\
=&\frac{\gamma_0 b}{4\pi h}\,\biggl[
\underbrace{\frac{\rho}h \tanh\left(\rho/h\right)}_{={\mathcal O}((\frac{\rho}h)^2)}-\log(\underbrace{\cosh(\rho/h)-1}_{=\frac{1}{2}(\rho/h)^2\,(1+{\mathcal O}((\rho/h)^2)})-\log 2)+O\left(-\frac{\rho}{b} \log \rho\right)\biggr]\\
=& \frac{\gamma_0 b}{2\pi h} \left[\log( h/\rho)+ O\left( (\rho/h)^2- \frac{\rho}{b} \log \rho\right)\right].
\end{align*}
This is eqn.~\eqref{eq.self}.
\end{proof}

\begin{proof}[Proof of~Proposition~\ref{prop.mainlowerbound}]

The kernel $W_\alpha$ has an analytical Fourier representation
\begin{equation*}
W_\alpha(x,y)=\sum_{k_2\neq0}\frac1{2\pi}\int\limits_{-\infty}^\infty
  \frac{(1-\alpha)k_1^2+(1+\alpha)k_2^2}{(k_1^2+k_2^2)^2}e^{i(k_1x+k_2y)}\,dk_1.
\end{equation*}
Notice that
\begin{equation*}
  -\frac{|x|}2=\frac1{2\pi}\int_{-\infty}^\infty
  \frac{\cos(k_1x)-1}{k_1^2}\,dk_1,
\end{equation*}
which can be interpreted as the summand for $k_2=0$. The Fourier transform is non-negative for $-1\leq\alpha\leq1$; for $\alpha=1$ the Fourier transform simplifies to ${\widehat W}_1(k) = \frac{k_2^2}{\|k\|^4}$.

Furthermore, we have the expression
\begin{equation*}
  W_\alpha(x,y)=\sum_{k_2=1}^\infty
  (1+\alpha k_2|x|)\frac{e^{-k_2|x|}}{k_2}\cos(k_2y),
\end{equation*}
which will be used in the sequel.

For $-1\leq\alpha\leq1$ we consider the function
\begin{equation}\label{eq:poisson}
    W_{\alpha,t}(x,y)=\sum_{n=1}^\infty\frac1n(1+\alpha n|x|)e^{-n|x|}\cos(ny)e^{-nt}.
\end{equation}
Then $W_{\alpha,t}$ is positive definite for all $t\geq0$ and $W_{\alpha,0}(x,y)=W_\alpha(x,y)$. Since~\eqref{eq:poisson} is similar to the geometric series it is possible to derive an analytical form
\begin{align*}
    W_{\alpha,t}(x,y)=&\frac{\alpha|x|\sinh(|x|+t)}{2(\cosh(|x|+t)-\cos(y))}
    -\frac12\log(2(\cosh(|x|+t)-\cos(y)))\\
    &+\frac t2+\frac{1-\alpha}2|x|.
\end{align*}
We will show that 
\begin{equation} \label{eq.Walphabd}
    W_\alpha(x,y)\geq W_{\alpha,t}(x,y)-\frac t2.
\end{equation}
by considering the function $W_{\alpha,t}(x,y)-\frac t2$.
We show that this function is monotonically decreasing in $t$ for all $(x,y)$. The derivative with respect to $t$ equals
\begin{multline*}
    \frac1{2(\cosh(|x|+t)-\cos(y))^2}\\\times\left(\alpha|x|(1-\cos(y)\cosh(|x|+t))-\sinh(|x|+t)(\cosh(|x|+t)-\cos(y))\right).
\end{multline*}
We have to show that the numerator is non-positive. We observe that the dependence on $\cos(y)$ is linear, so it suffices to show non-positivity for the two extremal values $\cos(y)=\pm1$. This gives the expressions
\begin{align*}
    &(\alpha|x|-\sinh(|x|+t))(1+\cosh(|x|+t))\\
    &(\alpha|x|+\sinh(|x|+t))(1-\cosh(|x|+t))
\end{align*}
which are both non-positive (for the range of $\alpha$), which proves~\eqref{eq.Walphabd}.

Then we argue
\begin{equation*}
    \sum_{i\neq j}W_\alpha(z_i-z_j)\geq
    \sum_{i,j}W_{\alpha,t}(z_i-z_j)-\frac t2N(N-1)-NW_{\alpha,t}(\mathbf{0}).
\end{equation*}
We have
\begin{align*}
    W_{\alpha,t}(\mathbf{0})=
    -\log\left(2\sinh\left(\frac t2\right)\right)\leq\log\left(\frac1t\right).
\end{align*}
Inserting $t=\frac2N$ into the above inequality and using the positive definiteness of $W_{\alpha,t}$ gives
\begin{align*}
    \sum_{i\neq j}W_\alpha(z_i-z_j)&\geq
    -(N-1)-N\log(N)+N \log 2\\
    &\geq -N\log N - (1-\log 2)N.
\end{align*}
This is~\eqref{eq.mainlb}.
\end{proof}
\begin{rmk}
    The singularity of the kernel $W_\alpha$ at $\mathbf{0}$ necessitates the exclusion of the diagonal terms from the sum \eqref{eq.mainlb}. This makes arguments using positive definiteness inapplicable directly. A classical method to remedy this problem is the use of a suitable smoothing operation; in our case the ''Poisson-type'' smoothing in \eqref{eq:poisson} turned out to be most suitable. After smoothing, positive definiteness can be applied. On the other hand the smoothing is responsible for the $\mathcal{O}(N)$ error term, which is conjectured to vanish.
\end{rmk}

\section{Continuous energies}\label{sec:cont-energ}
In this section we study continuous energies, namely expressions of the form
\begin{equation*}
    E_\alpha(\mu)=\iint\limits_{(\mathbb{R}\times\mathbb{T})^2}W_\alpha(z-z')\,\dd\mu(z)\,\dd\mu(z')
\end{equation*}
for all Borel probability measures $\mu$. The continuous energy is $\Gamma$-limit of the rescaled discrete energy for $N\to\infty$, this can be proved by following the argument in~\cite{Mora-Peletier-Scardia2017} (Theorem~3.3) where a related energy is derived from a semi-discrete strain energy model. Since its study turns out to be technically simpler than for the discrete case, it can also be seen as an idealisation. As reference for the interplay between the discrete and the continuous problem we refer to \cite{Borodachov_Hardin_Saff2019:discrete_energy}, see \cite{Landkof1972:potential_theory}  for a classical reference on potential theory.

In~\cite{Carrillo-Mateu-Mora-Rondi-Scardia-Verdera+2020:ellipse-law} it is shown that for $|\alpha|\leq 1$ and
$V_\alpha$ defined in~\eqref{eq.Valpha}
the minimizing probability measure of the energy 
$$ J_\alpha(\mu) = \int_{\R^2 \times \R^2} V_\alpha(z-z') \, \dd \mu(z)\, \dd \mu(z') + \int_{\R^2} |z|^2 \, \dd \mu(z)$$
is given by 
$$ \mu_\alpha = \frac{1}{\sqrt{1-\alpha^2}\pi} \chi_{\Omega(\sqrt{1-\alpha}, \sqrt{1+\alpha})},$$
where 
$$ \Omega(\sqrt{1-\alpha},\sqrt{1+\alpha}) = \left\{ (x,y) \in \R^2 \; : \; \frac{x^2}{1-\alpha}+\frac{y^2}{1+\alpha}<1 \right\}$$
is the characteristic function of an ellipse.
Theorem 1.1 in~\cite{Carrillo-Mateu-Mora-Rondi-Scardia-Verdera+2020:ellipse-law} provides the formula
$$ J_\alpha(\mu_\alpha) = C_\alpha +\frac12\int_{\R^2} |z|^2 \, \dd \mu_\alpha(z),$$
where
$$ C_\alpha = \frac12- \log\left(\frac{\sqrt{1-\alpha}+\sqrt{1+\alpha}}{2}\right) + \frac{\alpha\, \sqrt{1-\alpha}}{\sqrt{1-\alpha} +\sqrt{1+\alpha}}.$$
An easy calculation yields that $\int_{\R^2} |z|^2 \, \dd \mu_\alpha(z)= \frac{1}{2}$, this implies that $J_\alpha(\mu_\alpha)=J^\alpha_\mathrm{min}$, where $J^\alpha_{\min} = C_\alpha+\frac12$.


As opposed to the study of $\mathbb{R}^2$ as underlying space in \cite{Carrillo-Mateu-Mora-Rondi-Scardia-Verdera+2020:ellipse-law} we will investigate the periodized model on the space $\mathbb{R}\times\mathbb{T}$. This space turns out to share different features due to its nature as a Cartesian product of a non-compact and a compact space.
\subsection{Energy without external fields}\label{sec:cont-without}
For a measure $\mu$ on $\mathbb{R}\times\mathbb{T}$ we define the Fourier transform
\begin{equation*}
    \widehat{\mu}(k_1,k_2)=\iint_{\mathbb{R}\times\mathbb{T}}e^{-i(k_1x+k_2y)}\,\dd\mu(x,y)
\end{equation*}
for $(k_1,k_2)\in\mathbb{R}\times\mathbb{Z}$.
The energy $E_\alpha(\mu)$ can then be expressed as
\begin{equation*}
    E_\alpha(\mu)=\sum_{k_2\neq0}\frac1{2\pi}\int\limits_{-\infty}^\infty
  \frac{(1-\alpha)k_1^2+(1+\alpha)k_2^2}{(k_1^2+k_2^2)^2}|\widehat{\mu}(k_1,k_2)|^2\,\dd k_1,
\end{equation*}
which is non-negative and vanishes exactly, if
\begin{equation*}
    \widehat{\mu}=0 \quad \text{ on }\R \times (\Z \setminus \{0\}) 
\end{equation*}
which is equivalent to
\begin{equation*}
    \widehat{\mu}(k_1,k_2)=\widehat{\mu}(k_1,0)\widehat{\lambda}(k_2),
\end{equation*}
where $\lambda$ denotes the normalized Lebesgue measure on $\mathbb{T}$. Thus $E_\alpha(\mu)$ is exactly minimized for all measures of the form
$\mu=\nu\otimes\lambda$, where $\nu$ is any probability measure on $\mathbb{R}$.

In order to achieve a unique minimizer we can modify the energy in two ways. The first one is to modify the kernel by adding additional repulsion:
\begin{equation*}
    \widetilde{W}_{\alpha,\beta}(x,y)=W_{\alpha}(x,y)+\beta|x|
\end{equation*}
for $\beta>0$. The corresponding energy
\begin{equation}\label{eq:cont-no-ext}
    E_{\alpha,\beta}(\mu)=\iint\limits_{(\mathbb{R}\times\mathbb{T})^2}\widetilde{W}_{\alpha,\beta}(z-z')\,\dd\mu(z)\,\dd\mu(z')
\end{equation}
is then minimized for all measures of the form $\mu=\delta_{x_0}\otimes\lambda$ for fixed $x_0\in\mathbb{R}$ ($\delta_{x_0}$ denotes the point mass concentrated in $x_0$). This can be seen from the discussion above and the simple fact that
\begin{equation*}
    \iint\limits_{\mathbb{R}\times\mathbb{R}}|x_1-x_2|\,\dd\nu(x_1)\,\dd\nu(x_2)
\end{equation*}
is minimized exactly for point masses $\delta_{x_0}$.

Summing up, we have proved the following proposition.
\begin{prop}\label{prop:cont-no-ext}
    Let $-1\leq\alpha\leq1$. 
    For $\beta>0$ the  continuous energy \eqref{eq:cont-no-ext} is uniquely minimized by measures $\mu=\delta_{x_0}\otimes\lambda$ ($x_0\in\mathbb{R}$) amongst all Borel probability measures. For $\beta=0$ the energy is minimized for all measures $\mu=\nu\otimes\lambda$, where $\nu$ is any Borel probability measure in $\mathbb{R}$. For $\beta<0$ the energy takes arbitrarily large negative values and thus does not attain a minimum.
\end{prop}

The second possibility is to add an external field, which will be the subject of the next subsection.
\subsection{Energy with external fields}\label{sec:cont-with}
We consider the energy
\begin{equation*}
  E_\alpha^{\mathrm{ext}}(\mu)=\iint\limits_{(\mathbb{R}\times\mathbb{T})^2}
  \left(W_\alpha(z-z')+
    \frac{1-\alpha}2\left(|x|+|x'|\right)\right)\,\dd\mu(z)\,
  \dd\mu(z')
\end{equation*}
with $z=(x,y)$ for $-1\leq\alpha<1$. Then we have
\begin{equation*}
  E_\alpha^{\mathrm{ext}}(\mu)=E_\alpha(\mu)+
  \frac{1-\alpha}2\iint_{\mathbb{R}^2}\left(-|x-x'|+|x|+|x'|\right)\,
  \dd\nu(x)\,\dd\nu(x'),
\end{equation*}
where $\nu(A)=\mu(A\times\mathbb{T})$.

We study the two terms individually. The first term is the energy $E_\alpha(\mu)$ studied in Section~\ref{sec:cont-without}.
This is minimized with vanishing energy for all measures $\mu=\nu\otimes\lambda$.
The second term is
\begin{equation*}
  \frac{1-\alpha}2\iint_{\mathbb{R}^2}\left(-|x-x'|+|x|+|x'|\right)\,
  \dd\nu(x)\,\dd\nu(x')\geq0;
\end{equation*}
the positivity follows by the triangle inequality. Equality holds for $\nu=\delta_0$. In order to
prove that this is the unique minimizer, we assume that there exists $x_0\neq0$
such that $\forall \epsilon>0: \nu((x_0-\epsilon,x_0+\epsilon))>0$. We choose
$\epsilon<\frac{|x_0|}2$. Then we have
\begin{multline*}
  \iint_{(x_0-\epsilon,x_0+\epsilon)^2}\left(-|x-x'|+|x|+|x'|\right)\,
  \dd\nu(x)\,\dd\nu(x')\\\geq
  \iint_{(x_0-\epsilon,x_0+\epsilon)^2}\left(-2\epsilon+2(|x_0|-\epsilon)\right)
  \,\dd\nu(x)\,\dd\nu(x')\\=(2|x_0|-4\epsilon)\nu((x_0-\epsilon,x_0+\epsilon))>0,
\end{multline*}
which shows that for any measure $\nu$ with $0\neq x_0\in\mathrm{supp}(\nu)$ we
have strict inequality. Thus $\nu=\delta_0$ is the unique minimizer for the
second term.

Altogether, this shows that for $-1\leq\alpha<1$, $E_\alpha^{\mathrm{ext}}(\mu)$ is uniquely
minimized by $\mu=\delta_0\otimes\lambda$ amongst all Borel probability
measures.

More generally, we study
\begin{equation}\label{eq:cont-ext}
    E_{\alpha,\beta}^{\mathrm{ext}}(\mu)=\iint\limits_{(\mathbb{R}\times\mathbb{T})^2}
  \left(W_\alpha(z-z')+
    \beta\left(|x|+|x'|\right)\right)\,\dd\mu(z)\,
  \dd\mu(z')
\end{equation}
with $\beta>0$ and $-1\leq\alpha\leq1$. First, we notice that for $\beta>\frac{1-\alpha}2$, $E_{\alpha,\beta}(\mu)$ is uniquely minimized by $\mu=\delta_0\otimes\lambda$ by the discussion above. For $0<\beta<\frac{1-\alpha}2$ and $\alpha<1$ the energy $E_{\alpha,\beta}^{\mathrm{ext}}(\mu)$ attains arbitrarily large negative values. This can be seen by computing
\begin{equation*}
    \frac{1-\alpha}2\iint_{\mathbb{R}^2}\left(-|x_1-x_2|+\frac{2\beta}{1-\alpha}(|x|+|x'|)\right)\,\dd\nu(x)\,\dd\nu(x')
\end{equation*}
for $\dd\nu(x)=\frac{1+a}{2A^{a+1}}|x|^a\mathbbm{1}_{[-A,A]}(x)\,\dd x$, which gives
\begin{equation*}
  (1-\alpha)A\frac{a+1}{a+2}\left(-\frac{a+2}{2a+3}+\frac{2\beta}{1-\alpha}\right).
\end{equation*}
Since $\frac{a+2}{2a+3}$ approaches $1$ from below as $a\to-1+$, the term in
parenthesis can be made negative by choosing $a$ close enough to $-1$. Then the value can be made arbitrarily negative by increasing $A$.

Summing up, we have proved the following proposition.
\begin{prop}
    Let $-1\leq\alpha\leq1$. Then for $\beta>\frac{1-\alpha}2$ the energy \eqref{eq:cont-ext} is uniquely minimized by $\mu=\delta_0\otimes\lambda$. The same holds for $\alpha<1$ and $\beta=\frac{1-\alpha}2$. For $\beta<\frac{1-\alpha}2$ the energy attains arbitrarily large negative values and thus does not attain a minimum. 
\end{prop}

\begin{ackno}
    The authors are grateful to Felipe~Gon\c{c}alves and Stefan~Steinerberger for valuable discussions. The second author thanks Thomas~Hochrainer for providing a copy of \cite{Meurs2015:discrete-to-continuum}, which helped understanding the background from the point of view of material science.
\end{ackno}
\bibliographystyle{amsplain}
\bibliography{refs}

\end{document}